\documentclass[oneside,english]{amsart}
\usepackage[T1]{fontenc}
\usepackage[latin9]{inputenc}
\usepackage{amssymb}

\makeatletter
 \theoremstyle{plain}
\newtheorem{thm}{Theorem}[section]
  \theoremstyle{remark}
  \newtheorem*{acknowledgement*}{Acknowledgement}
  \theoremstyle{plain}
  \newtheorem{lem}[thm]{Lemma}
  \theoremstyle{plain}
  \newtheorem{prop}[thm]{Proposition}

\usepackage{alex}

\usepackage{babel}
\makeatother

\begin{document}

\title{Cuspidal representations which are not strongly cuspidal}

\author{Alexander Stasinski}

\address{DPMMS, University of Cambridge, Wilberforce Road, Cambridge, CB3
0WB, U.~K.}

\email{a.stasinski@dpmms.cam.ac.uk}

\begin{abstract}
We give a description of all the cuspidal representations of $\mbox{GL}_{4}(\mathfrak{o}_{2})$
in the sense of \cite{Aubert-Onn-Prasad}. This shows in particular
the existence of representations which are cuspidal, yet are not strongly
cuspidal, that is, do not have orbit with irreducible characteristic
polynomial mod $\mathfrak{p}$. As was shown in \cite{Aubert-Onn-Prasad},
this phenomenon cannot occur for $\mbox{GL}_{n}$, when $n$ is prime.
\end{abstract}
\maketitle
\begin{acknowledgement*}
This paper will appear in modified form as an appendix to a forthcoming
version of \cite{Aubert-Onn-Prasad}. I wish to thank U. Onn for comments
which led to simplifications of the exposition, and for explaining
the idea that cuspidal representations of $\mbox{GL}_{2}(\mathbb{F}_{q^{2}})$
should give rise to cuspidal representations of $\mbox{GL}_{4}(\mathfrak{o}_{2})$
in the function field case.
\end{acknowledgement*}

\section{Preliminaries and reductions}

Recall the notation of \cite{Aubert-Onn-Prasad} regarding the groups
$G_{\lambda}$, and the definitions of geometric/infinitesimal induction,
and cuspidality. We consider an arbitrary local field $F$ with ring
of integers $\mathfrak{o}$, maximal ideal $\mathfrak{p}$, and finite
residue field $\mathbb{F}_{q}$. Let $n=4$ and $k=2$, and put $G:=G_{2^{4}}\cong\mbox{GL}_{4}(\mathfrak{o}_{2})$,
where $\mathfrak{o}_{2}=\mathfrak{o}/\mathfrak{p}^{2}$. If $\pi$
is a cuspidal representation of $G$, then by \cite{Aubert-Onn-Prasad},
Proposition 4.4 it is primary, that is, its orbit in $M_{4}(\mathbb{F}_{q})$
consists of matrices whose characteristic polynomial is of the form
$f(X)^{n}$, where $f(X)$ is an irreducible polynomial. If $n=1$,
then $\pi$ is strongly cuspidal (by definition), and such representations
were described in \cite{Aubert-Onn-Prasad}, Sect.~5. On the other
hand, $f(X)$ cannot have degree 1, because then it would be infinitesimally
induced from $G_{(2,1^{3})}$, up to $1$-dimensional twist (cf. the
end of the proof of Theorem 4.3 in \cite{Aubert-Onn-Prasad}). We
are thus reduced to considering representations whose characteristic
polynomial is a reducible power of a non-linear irreducible polynomial.
In the situation we are considering, there is only one such possibility,
namely the case where $f(X)$ is quadratic, and $n=2$. Let $\eta$
denote an element which generates the extension $\mathbb{F}_{q^{2}}/\mathbb{F}_{q}$.
We consider $M_{2}(\mathbb{F}_{q^{2}})$ as embedded in $M_{4}(\mathbb{F}_{q})$
via the embedding $\mathbb{F}_{q^{2}}\hookrightarrow M_{2}(\mathbb{F}_{q})$,
by choosing the basis $\{1,\eta\}$ for $\mathbb{F}_{q^{2}}$ over
$\mathbb{F}_{q}$. Rational canonical form implies that in $M_{4}(\mathbb{F}_{q})$
there are two conjugation orbits with two irreducible $2\times2$
blocks, one regular, and one which is not regular (we shall call the
latter \emph{irregular}), represented by the following elements, respectively:\[
\beta_{1}=\begin{pmatrix}\eta & 1\\
0 & \eta\end{pmatrix},\qquad\beta_{2}=\begin{pmatrix}\eta & 0\\
0 & \eta\end{pmatrix},\]
Therefore, any irreducible cuspidal non-strongly cuspidal representation
of $G$ has exactly one of the elements $\beta_{1}$ or $\beta_{2}$
in its orbit.

Denote by $K_{1}$ the kernel of the reduction map $G=G_{2^{4}}\rightarrow G_{1^{4}}$.
In the following we will let $\psi$ be a fixed non-trivial additive
character on $\mathfrak{o}$ with conductor $\mathfrak{p}^{2}$. Then
for each $\beta\in M_{4}(\mathbb{F}_{q})$ we have a character $\psi_{\beta}:K_{1}\rightarrow\mathbb{C}^{\times}$
defined by \[
\psi_{\beta}(x)=\psi(\mathrm{Tr}(\beta(x-1))).\]
The group $G$ acts on its normal subgroup $K_{1}$ via conjugation,
and thus on the set of characters of $K_{1}$ via the {}``coadjoint
action''. For any character $\psi_{\beta}$ of $K_{1}$, we write
\[
G(\psi_{\beta}):=\Stab_{G}(\psi_{\beta}).\]
By Proposition 2.3 in \cite{Hill_regular}, the stabilizer $G(\psi_{\beta})$
is the preimage of the centralizer $C_{G_{1^{4}}}(\beta)$, under
the reduction mod $\mathfrak{p}$ map. 

By definition, an irreducible representation $\pi$ of $G$ is cuspidal
iff none of its 1-dimensional twists $\pi\otimes\chi\circ\det$ has
any non-zero vectors fixed under any group $U_{i,j}$ or $U_{\lambda\hookrightarrow2^{4}}$,
or equivalently (by Frobenius reciprocity), if $\pi\otimes\chi\circ\det$
does not contain the trivial representation $\mathbf{1}$ when restricted
to $U_{i,j}$ or $U_{\lambda\hookrightarrow2^{4}}$. The groups $U_{i,j}$
are analogs of unipotent radicals of (proper) maximal parabolic subgroups
of $G$, and $U_{\lambda\hookrightarrow2^{4}}$ are the infinitesimal
analogs of unipotent radicals (cf. \cite{Aubert-Onn-Prasad}, Sect.~3).
Note that since $\Ind_{U_{i,j}}^{G}\mathbf{1}=\Ind_{U_{i,j}}^{G}(\mathbf{1}\otimes\chi\circ\det)=(\Ind_{U_{i,j}}^{G}\mathbf{1})\otimes\chi\circ\det$,
for any character $\chi:\mathfrak{o}_{2}^{\times}\rightarrow\mathbb{C}^{\times}$,
a representation is a subrepresentation of a geometrically induced
representation if and only if all its one-dimensional twists are.

In our situation, that is, for $n=4$ and $k=2$, there are three
distinct geometric stabilizers, $P_{1,3}$, $P_{2,2}$, and $P_{3,1}$
with {}``unipotent radicals'' $U_{1,3}$, $U_{2,2}$, and $U_{3,1}$,
respectively. Thus a representation is a subrepresentation of a geometrically
induced representation if and only if it is a component of $\Ind_{U_{i,j}}^{G}\mathbf{1}$,
for some $(i,j)\in\{(1,3),(2,2),(3,1)\}$. Furthermore, there are
three partitions, written in descending order, which embed in $2^{4}$
and give rise to non-trivial infinitesimal induction functors, namely\[
(2,1^{3}),\ (2^{2},1^{2}),\ (2^{3},1).\]
Thus a representation is a subrepresentation of an infinitesimally
induced representation if and only if it is a component of $\Ind_{U_{\lambda\hookrightarrow2^{4}}}^{G}\mathbf{1}$,
for some partition $\lambda$ as above. Because of the inclusions\[
U_{(2,1^{3})\hookrightarrow2^{4}}\subset U_{(2^{2},1^{2})\hookrightarrow2^{4}}\subset U_{(2^{3},1)\hookrightarrow2^{4}},\]
an irreducible representation of $G$ is a component of an infinitesimally
induced representation if and only if it is a component of $\Ind_{U_{(2,1^{3})\hookrightarrow2^{4}}}^{G}\mathbf{1}$.

\begin{lem}
\label{lem:first}Suppose that $\pi$ is an irreducible representation
of $G$ whose orbit contains either $\beta_{1}$ or $\beta_{2}$.
Then $\pi$ is not an irreducible component of any representation
geometrically induced from $P_{1,3}$ or $P_{3,1}$. Moreover, no
$1$-dimensional twist of $\pi$ is an irreducible component of an
infinitesimally induced representation.
\end{lem}
\begin{proof}
If $\pi$ were a component of $\Ind_{U_{1,3}}^{G}\mathbf{1}$, then
$\langle\pi|_{U_{1,3}},\mathbf{1}\rangle\neq0$, so in particular
$\langle\pi|_{K_{1}\cap U_{1,3}},\mathbf{1}\rangle\neq0$, which implies
that $\pi|_{K_{1}}$ contains a character $\psi_{b}$, where $b=(b_{ij})$
is a matrix such that $b_{i1}=0$ for $i=2,3,4$. This means that
the characteristic polynomial of $b$ would have a linear factor,
which contradicts the hypothesis. The case of $U_{3,1}$ is treated
in exactly the same way, except that the matrix $b$ will have $b_{4j}=0$
for $j=1,2,3$. The case of infinitesimal induction is treated using
the same kind of argument. Namely, if $\pi$ were a component of $\Ind_{U_{(2,1^{3})\hookrightarrow2^{4}}}^{G}\mathbf{1}$,
then $U_{(2,1^{3})\hookrightarrow2^{4}}\subset K_{1}$ and $\langle\pi|_{U_{(2,1^{3})\hookrightarrow2^{4}}},\mathbf{1}\rangle\neq0$,
which implies that $\pi|_{K_{1}}$ contains a character $\psi_{b}$,
where $b=(b_{ij})$ is a matrix such that $b_{1j}=0$ for $j=1,\dots,4$.
A $1$-dimensional twist of $\pi$ would then contain a character
$\psi_{aI+b}$, where $a$ is a scalar and $I$ is the identity matrix.
The matrix $aI+b$ has a linear factor in its characteristic polynomial,
which contradicts the hypothesis.
\end{proof}
We now consider in order representations whose orbits contain $\beta_{1}$
or $\beta_{2}$, respectively. In the following we will write $\bar{P}_{2,2}$
and $\bar{U}_{2,2}$ for the images mod $\mathfrak{p}$ of the groups
$P_{2,2}$ and $U_{2,2}$, respectively.

\section{The regular cuspidal representations}

Assume that $\pi$ is an irreducible representation of $G$ whose
orbit contains $\beta_{1}$. Since $\beta_{1}$ is a regular element,
the representation $\pi$ can be constructed explicitly as an induced
representation (cf. \cite{Hill_regular}). In particular, it is shown
in \cite{Hill_regular} that there exists a $1$-dimensional representation
$\rho$ of $G(\psi_{\beta_{1}})$ (uniquely determined by $\pi$)
such that $\rho|_{K_{1}}=\psi_{\beta_{1}}$, and such that\[
\pi=\Ind_{G(\psi_{\beta_{1}})}^{G}\rho.\]

\begin{prop}
\label{pro:cuspidal beta1}The representation $\pi$ is cuspidal if
and only if $\rho$ does not contain the trivial representation of
$G(\psi_{\beta_{1}})\cap U_{2,2}$.
\end{prop}
\begin{proof}
Lemma \ref{lem:first} shows that $\pi$ is cuspidal if and only if
it is not a component of $\Ind_{U_{2,2}}^{G}\mathbf{1}$. By Mackey's
intertwining number theorem (cf. \cite{Curtis_Reiner}, 44.5), we
have\[
\langle\pi,\Ind_{U_{2,2}}^{G}\mathbf{1}\rangle=\langle\Ind_{G(\psi_{\beta_{1}})}^{G}\rho,\Ind_{U_{2,2}}^{G}\mathbf{1}\rangle=\sum_{x\in G(\psi_{\beta_{1}})\backslash G/U_{2,2}}\langle\rho|_{G(\psi_{\beta_{1}})\cap\leftexp{x}{U_{2,2}}},\mathbf{1}\rangle,\]
so this number is zero if and only if $\langle\rho|_{G(\psi_{\beta_{1}})\cap\leftexp{x}{U_{2,2}}},\mathbf{1}\rangle=0$
for each $x\in G$. Assume that $\pi$ is cuspidal. Then in particular,
taking $x=1$, we have $\langle\rho|_{G(\psi_{\beta_{1}})\cap U_{2,2}},\mathbf{1}\rangle=0$.

Conversely, assume that $\pi$ is not cuspidal. Then $\langle\rho|_{G(\psi_{\beta_{1}})\cap^{x}U_{2,2}},\mathbf{1}\rangle\neq0$,
for some $x\in G$, and in particular, $\langle\rho|_{K_{1}\cap\leftexp{x}{U_{2,2}}},\mathbf{\mathbf{1}}\rangle=\langle\psi_{\beta_{1}}|_{K_{1}\cap\leftexp{x}{U_{2,2}}},\mathbf{1}\rangle\neq0$.
Write $\bar{x}$ for $x$ modulo $\mathfrak{p}$. Now $\psi_{\beta_{1}}|_{K_{1}\cap\leftexp{x}{U_{2,2}}}=\psi_{\beta_{1}}|_{\leftexp{x}{(K_{1}\cap U_{2,2})}}$,
and $\psi_{\beta_{1}}(\leftexp{x}{g})=\psi_{\bar{x}^{-1}\beta_{1}\bar{x}}(g)$,
for any $g\in K_{1}\cap U_{2,2}$. Let $\bar{x}^{-1}\beta_{1}\bar{x}$
be represented by the matrix\[
\begin{pmatrix}A_{11} & A_{12}\\
A_{21} & A_{22}\end{pmatrix},\]
where each $A_{ij}$ is a $2\times2$-block. Then from the definition
of $\psi_{\bar{x}^{-1}\beta_{1}\bar{x}}$ and the condition $\psi_{\bar{x}^{-1}\beta_{1}\bar{x}}(g)=1$,
for all $g\in K_{1}\cap U_{2,2}$, it follows that $A_{21}=0$; thus
\[
\bar{x}^{-1}\beta_{1}\bar{x}\in\bar{P}_{2,2}.\]
Since $\bar{x}^{-1}\beta_{1}\bar{x}$ is a block upper-triangular
matrix with the same characteristic polynomial as $\beta_{1}$, we
must have $A_{11}=B_{1}\eta B_{1}^{-1}$, $A_{22}=B_{2}\eta B_{2}^{-1}$,
for some $B_{1},B_{2}\in\mbox{GL}_{2}(\mathbb{F}_{q})$. Then there
exists $p\in\bar{P}_{2,2}$ such that \[
(\bar{x}p)^{-1}\beta_{1}(\bar{x}p)=\begin{pmatrix}\eta & B\\
0 & \eta\end{pmatrix},\]
for some $B\in M_{2}(\mathbb{F}_{q})$ (in fact, we can take $p=\left(\begin{smallmatrix}B_{1}^{-1} & 0\\
0 & B_{2}^{-1}\end{smallmatrix}\right)$). Levi decomposition of $\beta_{1}$ and $(\bar{x}p)^{-1}\beta_{1}(\bar{x}p)$
implies that the semisimple parts $(\bar{x}p)^{-1}\left(\begin{smallmatrix}\eta & 0\\
0 & \eta\end{smallmatrix}\right)(\bar{x}p)$ and $\left(\begin{smallmatrix}\eta & 0\\
0 & \eta\end{smallmatrix}\right)$ are equal, that is, $\bar{x}p\in C_{G_{1^{4}}}(\left(\begin{smallmatrix}\eta & 0\\
0 & \eta\end{smallmatrix}\right))=\mbox{GL}_{2}(\mathbb{F}_{q^{2}})$. Now, in $\mbox{GL}_{2}(\mathbb{F}_{q^{2}})$, the equation $(\bar{x}p)^{-1}\beta_{1}(\bar{x}p)=\left(\begin{smallmatrix}\eta & B\\
0 & \eta\end{smallmatrix}\right)$ implies that $\bar{x}p\in\left(\begin{smallmatrix}* & *\\
0 & *\end{smallmatrix}\right)\subset\bar{P}_{2,2}$, so $\bar{x}\in\bar{P}_{2,2}$, and hence $x\in K_{1}P_{2,2}$. The
facts that $U_{2,2}$ is normal in $P_{2,2}$, and that $\langle\rho|_{G(\psi_{\beta_{1}})\cap\leftexp{x}{U_{2,2}}},\mathbf{1}\rangle$
only depends on the right coset of $x$ modulo $K_{1}$ then imply
that \[
0\neq\langle\rho|_{G(\psi_{\beta_{1}})\cap\leftexp{x}{U_{2,2}}},\mathbf{1}\rangle=\langle\rho|_{G(\psi_{\beta_{1}})\cap U_{2,2}},\mathbf{1}\rangle.\]

\end{proof}
The preceding proposition shows that we can construct all the cuspidal
representations of $G$ with orbit containing $\beta_{1}$ by constructing
the corresponding $\rho$ on $G(\psi_{\beta_{1}})$. Since $\psi_{\beta_{1}}$
is trivial on $K_{1}\cap U_{2,2}$, we can extend $\psi_{\beta_{1}}$
to a character of $(G(\psi_{\beta_{1}})\cap U_{2,2})K_{1}$, trivial
on $G(\psi_{\beta_{1}})\cap U_{2,2}$. Then $\psi_{\beta_{1}}$ can
be extended to a character $\tilde{\psi}_{\beta_{1}}$ on the whole
of $G(\psi_{\beta_{1}})$, such that $\tilde{\psi}_{\beta_{1}}$ is
trivial on $G(\psi_{\beta_{1}})\cap U_{2,2}$ (this incidentally shows
that there exist irreducible non-cuspidal representations of $G$
whose orbit contains $\beta_{1}$). Now let $\theta$ be a representation
of $G(\psi_{\beta_{1}})$ obtained by pulling back a representation
of $G(\psi_{\beta_{1}})/K_{1}$ that is non-trivial on $(G(\psi_{\beta_{1}})\cap U_{2,2})K_{1}/K_{1}$.
Then $\rho:=\theta\otimes\tilde{\psi}_{\beta_{1}}$ is non-trivial
on $G(\psi_{\beta_{1}})\cap U_{2,2}$, and all such representations
are obtained for some $\theta$ as above. 

Proposition \ref{pro:cuspidal beta1} shows that there is a canonical
1-1 correspondence between irreducible representations of $G(\psi_{\beta_{1}})$
which contain $\psi_{\beta_{1}}$ and are non-trivial on $G(\psi_{\beta_{1}})\cap U_{2,2}$,
and cuspidal representations of $G$ with $\beta_{1}$ in their respective
orbits. We shall now extend this result to cuspidal representations
which have $\beta_{2}$ in their respective orbits, and thus cover
all cuspidal representations of $G$.

\section{The irregular cuspidal representations}

Assume now that $\pi$ is an irreducible representation of $G$ whose
orbit contains $\beta_{2}$. Although $\beta_{2}$ is not regular,
it is strongly semisimple in the sense of \cite{Hill_semisimple_cuspidal},
Definition 3.1, and thus $\pi$ can be constructed explicitly in a
way similar to the regular case. More precisely, Proposition 3.3 in
\cite{Hill_semisimple_cuspidal} implies that there exists an irreducible
representation $\tilde{\psi}_{\beta_{2}}$ of $G(\psi_{\beta_{2}})$,
such that $\tilde{\psi}_{\beta_{2}}|_{K_{1}}=\psi_{\beta_{2}}$, and
any extension of $\psi_{\beta_{2}}$ to $G(\psi_{\beta_{2}})$ is
of the form $\rho:=\theta\otimes\tilde{\psi}_{\beta_{2}}$, for some
irreducible representation $\theta$ pulled back from a representation
of $G(\psi_{\beta_{2}})/K_{1}$. Then \[
\pi=\Ind_{G(\psi_{\beta_{2}})}^{G}\rho\]
is an irreducible representation, any representation of $G$ with
$\beta_{2}$ in its orbit is of this form, and as in the regular case,
$\rho$ is uniquely determined by $\pi$. We then have a result completely
analogous to the previous proposition:

\begin{prop}
The representation $\pi$ is cuspidal if and only if $\rho$ does
not contain the trivial representation of $G(\psi_{\beta_{2}})\cap U_{2,2}$.
\end{prop}
\begin{proof}
The proof of Proposition \ref{pro:cuspidal beta1} with $\beta_{1}$
replaced by $\beta_{2}$, goes through up to the point where (under
the assumption that $\pi$ is not cuspidal) we get $\bar{x}p\in C_{G_{1^{4}}}(\left(\begin{smallmatrix}\eta & 0\\
0 & \eta\end{smallmatrix}\right))=G(\psi_{\beta_{2}})/K_{1}$. It then follows that $x\in G(\psi_{\beta_{2}})P_{2,2}$, and since
$U_{2,2}$ is normal in $P_{2,2}$, and $\langle\rho|_{G(\psi_{\beta_{2}})\cap\leftexp{x}{U_{2,2}}},\mathbf{1}\rangle$
only depends on the right coset of $x$ modulo $G(\psi_{\beta_{2}})$,
we get \[
0\neq\langle\rho|_{G(\psi_{\beta_{2}})\cap\leftexp{x}{U_{2,2}}},\mathbf{1}\rangle=\langle\rho|_{G(\psi_{\beta_{2}})\cap U_{2,2}},\mathbf{1}\rangle.\]

\end{proof}
\bibliographystyle{alex}
\bibliography{alex}

\end{document}